\documentclass[11pt]{amsart}

\usepackage{amssymb,latexsym}
\usepackage{enumitem}
\usepackage{comment}
\usepackage{bbm}
\usepackage{bm}

\makeatletter

\@namedef{subjclassname@2010}{

  \textup{2010} Mathematics Subject Classification}
\usepackage{amssymb, amsmath,amsthm}
\usepackage{hyperref}
\newtheorem{thm}{Theorem}[section]

\newtheorem{conj}[thm]{Conjecture}

\newtheorem{cor}[thm]{Corollary}
\newtheorem{lem}[thm]{Lemma}
\theoremstyle{definition}
\newtheorem{rem}[thm]{Remark}

\newcommand{\mc}{\mathcal}
\newcommand{\mf}{\mathfrak}
\newcommand{\mb}{\mathbb}

\newcommand{\R}{\mb{R}}
\newcommand{\C}{\mb{C}}
\newcommand{\N}{\mb{N}}
\newcommand{\Z}{\mb{Z}}

\newcommand{\lb}{\label}
\newcommand{\ef}{\eqref}
\newcommand{\leg}[2]{\genfrac{(}{)}{}{}{#1}{#2}}
\renewcommand{\bar}{\overline}

\newcommand{\bv}{\Bigl(}
\newcommand{\bn}{\Bigr)}
 
\newenvironment{acknowledgements} {\begin{abstract}} {\end{abstract}}

\newcommand{\proofpart}[2]{%
  \par
  \addvspace{\medskipamount}%
  \noindent\emph{Part #1: #2}\par\nobreak
  \addvspace{\smallskipamount}%
  \@afterheading
}

\title{On Completely multiplicative $\pm1$ sequences that omit many consecutive $+1$ values}
\author{Yichen You}
\address{Department of Mathematical Sciences, Durham University, Stockton Road, Durham, DH1 3LE, UK}
\email{yichen.you@outlook.com}

\begin{document}
\begin{abstract}
   We investigate the construction of $\pm1$-valued completely multiplicative functions that take the value $+1$ at at most $k$ consecutive integers, which we call length-$k$ functions. We introduce a way to extend the length based on the idea of the "rotation trick" \cite{klurman2021multiplicative} and such an extension can be quantified by the number of modified primes. Under the assumption of Elliott's conjecture, this method allows us to construct length-$k$ functions systematically for $k\geq 4$ which generalizes the work of I. Schur for $k = 2$ and R. Hudson for $k =3$.  
\end{abstract}
\maketitle

\section{Introduction}
We say $f: \N \to \C$ is \textit{completely multiplicative} if $f(ab) = f(a)f(b)$ for all $a, b \in \N$. Let 
\begin{align*}
    \mc{M}' &=\{f:\N \to \{+1, -1\}: \text{ $f$ is completely multiplicative}\}.
\end{align*}
Recent works in the theory of completely multiplicative functions have proved fruitful in understanding the extremal behaviour of bounded sequences, as in the Erd\H{o}s Discrepancy Problem. In 1957, P. Erd\H{o}s \cite{erdHos1957some} conjectured that for any sequence $\{f(n)\}_{n \in \N}$ consisting of $\pm1$, there is $x, d \in \N$ such that 
\begin{align*}
    \sup\limits_{x, d}\Big|\sum\limits_{n \leq x}f(nd)\Big| = \infty.
\end{align*}
In 2015, T. Tao \cite{tao2016erdHos} proved it by reducing $\{f(n)\}_{n \in \N}$ to $f \in \mc{M}'$. Moreover, his result indicated that every $f \in \mc{M}'$ satisfies 
\begin{align}\lb{ttao}
    \limsup\limits_{x \to \infty}|\sum\limits_{n \leq x}f(n)| = \infty.
\end{align}
One can find more works on classifying functions related to \ef{ttao} in \cite{aymone2022erdHos},\cite{klurman2021multiplicative}, \cite{klurman2018rigidity} and \cite{klurman2017correlations}.

Our main focus is about the pattern of  $\{f(n)\}_{n \in \N}$ for $f \in \mc{M}'$. 
We say that $f\in \mc{M}'$ is a \textit{length-}$k$ function if $k$ is the largest positive integer such that there exists $n \in \N$ for which $f(n+1)= \cdots= f(n+k) = +1$. It is natural to try to classify all $f \in \mc{M}'$ of length $k$, for each $k\geq 2$. \footnote{There is no length-$1$ function, as either $f(1)=f(2)=+1$, or $f(4)=f(5) = +1$, or $f(9)=f(10)=+1$.} Motivated by the results of D. H. Lehmer and E. Lehmer \cite{lehmer1962runs} on the first appearance of consecutive quadratic residues, W. H. Mills conjectured \cite{mills1965bounded} that there are only two length-$2$ functions $f_1, f_2$, for which
\begin{align*}
    f_i(p) = 
    \begin{cases}
    \leg{p}{3} &\text{ if } p\, \nmid \,3,\\
    (-1)^i &\text{ if } p = 3
    \end{cases}
\end{align*}
where $i = 1, 2$ and $\leg{\cdot}{3}$ is the Legendre symbol mod $3$, and I. Schur \cite{schur1973multiplicative} confirmed this. For $k = 3$, R. Hudson \cite{hudson1974totally} conjectured that there are only thirteen possibilities. This has recently been proved by O. Klurman, A. P. Mangerel, and J. Teräväinen \cite{klurman2023elliott}. 
\begin{thm}[Formerly Hudson's conjecture {\cite[Theorem 2.4]{klurman2023elliott}}]
Let $q \in \{5, 7, 11, 13, 53\}$ and $i = 1, 2$. Define
\begin{align*}
     f_{(q, i)}(p) &= 
    \begin{cases}
    \leg{p}{q} &\text{ if } p\, \nmid \,q,\\
    (-1)^i &\text{ if } p = q,
    \end{cases} \;\;\;\;
    f_{(4, i)}(p) = 
    \begin{cases}
    \leg{p}{4} &\text{ if } p\, \nmid \,4,\\
    (-1)^i &\text{ if } p = 2,
    \end{cases}\;\;\;\;
    g(p) = 
    \begin{cases}
    1 &\text{ if } p\, \nmid \,2,\\
    -1 &\text{ if } p = 2.
    \end{cases}
\end{align*}
If $f:\N \to \{+1, -1\}$ is a length-3 function, then $f$ must be one of the above.
\end{thm}

In light of Hudson's conjecture, we want to classify functions of higher length, say $k \geq 4$. By the above, we learn that there are finitely many length-$2$ and length-$3$ functions. For $k \geq 4$, conversely, it is possible to construct infinitely many examples. For instance, we can construct a length-$4$ function $f$ by
\begin{align*}
    f(p) = 
    \begin{cases}
    \leg{p}{5} &\text{ if } p\, \nmid \,5 \text{ and } q,\\
    -\leg{p}{5} &\text{ if } p = q,\\
    1 &\text{ if } p = 5,
    \end{cases} 
\end{align*}
where $q$ can be any prime with $q \equiv 2 $ mod $5$. As there are infinitely many choices of $q$, we can construct infinitely many length-$4$ functions $f$. Nevertheless, we believe that such counterexamples can be constructed in a systematic way. 

Let $\chi_q$ be a real character mod $q$. We define a \textit{modified character} $\tilde{\chi}_q \in \mc{M}'$ at each prime by
\begin{align*}
    \tilde{\chi}_q(p) = 
    \begin{cases}
        \chi(p) \;\; \text{if $p \nmid q$},\\
        \eta(p) \;\; \text{if $p | q$, where $\eta(p) \in \{+1, -1\}$} .
    \end{cases}
\end{align*}
Our main result shows that the length of $\tilde{\chi}_q$ can be extended to at least $k$ by altering its values at a set of finitely many prime numbers $p > k$ and $p \nmid q$, which we call \textit{modified primes}, whose size is covered by Theorem \ref{thm1.2}. We will use $\delta_q(k)$ to denote the minimal number of modified primes that are needed for the length of $\tilde{\chi}_q$ to be extended to at least $k$. 
\begin{thm}\lb{thm1.2}
    Let $k$ be a positive integer and $\tilde{\chi}_q$ be a modified character mod $q > 1$. Then we have $\delta(k) = \max\limits_{q\in \N}\delta_q(k) = \frac{1}{2}k + O(\log k)$.
\end{thm}

Here we briefly explain how the modified primes can extend the length of $\tilde{\chi}_q$. Let $I=\{a+1, \ldots, a+k\} \subset \N $ be an interval of length $k$ and $(\tilde{\chi}_q(a+1), \ldots, \tilde{\chi}_q(a+k))$ be the sign pattern of $\tilde{\chi}_q$ on $I$, denoted by $\tilde{\chi}_q(I)$. We assume that the length of $\tilde{\chi}_q$ is less than $k$, i.e., $\tilde{\chi}_q(I) \neq (+1, \ldots, +1)$ for all $I$. Since $\tilde{\chi}_q$ inherits some amount of periodicity from $\chi_q$, by choosing an appropriate $a$ according to the choice of modified primes, we can find an interval $I'$ of length $k$ that for all $n \in I'$ with $\tilde{\chi}_q(n) = -1$, there is a unique $p > k, p \nmid q$ such that $p^\nu \parallel n$ for some odd integer $\nu > 0$. By modifying $\tilde{\chi}_q(p)$ to $-\tilde{\chi}_q(p)$ for all modified primes $p$, all the values of $-1$ in $\tilde{\chi}_q(I')$ will turn to $+1$, and hence the length of $\tilde{\chi}_q$ is extended to at least $k$. Lemma \ref{lem2.1} is based on this idea.  

Before giving the application of the extension to $f \in \mc{M}'$, we will introduce some concepts from the pretentious approach to analytic number theory of Granville and Soundararajan \cite{Pretentiousapproach}. Let $\mc{M}$ denote the set of multiplicative functions $f: \N \to \C$ with $|f(n)| \leq 1$ for all $n\in \N$. Given $x \geq 1$, the \textit{pretentious distance} between $f, g \in \mc{M}$ is
\begin{equation*}
    \mb{D}(f, g;x) = \Bigl(\sum\limits_{p\leq x}\frac{1-\mf{Re}(f(p)\bar{g(p)})}{p}\Bigr)^{1/2}.
\end{equation*}
This satisfies the triangle inequality:
\begin{equation}\lb{tri}
    \mb{D}(f_1, h_1;x) + \mb{D}(f_2, h_2;x) \geq \mb{D}(f_1f_2, h_1h_2;x) \text{ for } f_1, f_2, h_1, h_2 \in \mc{M}.
\end{equation}
Let $f, g \in \mc{M}$. We say that $f$ is \textit{pretentious} to $g$ if
\begin{align*}
   \mb{D}(f, g;x) = O(1) \text{ as } x \to \infty.
\end{align*}
We say furthermore that $f$ is a \textit{pretentious function} if $f$ is pretentious to a twisted character $\chi(n)n^{it}$ where $\chi(n)$ is a Dirichlet character and $t \in \R$, otherwise, it is a \textit{non-pretentious function}. Additionally, if $f \in \mc{M}$ is real, $f$ can only be pretentious to a real character (in the proof of Corollary 1.4).

\begin{conj}[Elliott's conjecture]\lb{conj1.3}
  For any fixed $a_i, b_i, N \in \N$ such that $a_ib_j \neq a_jb_i$ for all $i, j = 1, 2, ..., N \text{ and } i \neq j$, if $f\in\mc{M}'$ is non-pretentious \footnote{The general condition for $f \in \mc{M}$ is $\inf\limits_{\chi, |t| \leq x} \mb{D}(f, \chi(n)n^{it}; x)^2 \to \infty$ as $x \to \infty$. }, then
\begin{equation}\lb{Elliott}
    \sum\limits_{n \leq x}\prod\limits_{i = 1}^N f(a_in+b_i) = o(x).
\end{equation}  
\end{conj}
Under the assumption of Conjecture \ref{conj1.3}, any $f \in \mc{M}'$ of finite length is pretentious to a real primitive (not necessarily non-principal) character $\chi_q$ (the detailed proof will be in the next section). As a consequence, the extension of $f$ of finite length can be reduced to the extension of a corresponding modified character $\tilde{\chi}_q$.
\begin{cor}\lb{cor1.4}
    Given $k \in \N$, let $f$ be a length-$k$ function. 
    \begin{itemize}
        \item[(a)] Assuming Conjecture \ref{conj1.3}, there is a real character $\chi_q$ mod $q>1$ such that $f$ is pretentious to $\chi_q$.
        \item[(b)] Suppose $f$ is pretentious to a real character $\chi_q$ mod $q>1$. Let 
        \begin{align} \lb{defofJ}
            \mc{J}(k) = \{p\in \mc{P}: f(p) \neq \chi_q(p) \text{ with } p \nmid q\text{ and } p > k\}.
        \end{align}
    Then $|\mc{J}(k)| \leq \frac{1}{2}k$.
    \end{itemize}
\end{cor}
By Lemma \ref{lem2.1}, we can see that the length increases with the number of modified primes, and hence Corollary \ref{cor1.4} follows from Theorem \ref{thm1.2} immediately.
\begin{rem}
    In our extension, we only consider modified primes $ p > k$. It remains to investigate a more general case where the modified primes involve primes $p \leq k$, i.e., modified primes can be any primes, which allows one to construct length-$k$ functions in a more general way. Unlike modified primes $ p > k$, the number of flipped values of $f(n)$ in an interval of length $k$ by a modified prime $p \leq k$ is more than one. Therefore, it would be fairly tricky to determine $\delta(k)$. If we include $p \leq k$ in the count given by $\delta(k)$ then we can currently only bound it crudely by $\frac{k}{\log k}$ using the prime number theorem.
\end{rem}

\section{Proof of Theorem \ref{thm1.2} and Corollary \ref{cor1.4}} 
We will use auxiliary Lemma \ref{lem2.1} and Lemma \ref{lem2.2} to prove Theorem \ref{thm1.2}. Let $S \subset \mc{P}$ be a subset of prime numbers. We define $\lambda_{S}\in \mc{M}'$ at each prime by
\begin{align*}
    \lambda_{S}(p) =
    \begin{cases}
        1 &\text{ if } p \in \mc{P}-S,\\
        -1 &\text{ if } p \in S.
    \end{cases}    
\end{align*}

\begin{lem}\lb{lem2.1}
    Let $k \in \N$ and $I_k$ be an interval of length $k$. Let $\tilde{\chi}_q$ be a modified character mod $q$ and $r$ denote the number of integers $n \in I_k$ where $\tilde{\chi}_q(n) = -1$. Let $P_r = \{p_1, \ldots, p_r\}$ be a set of $r$ distinct primes with $p_i>k, p_i \nmid q$ for $i = 1,\ldots, r$. We define $f(n) = \tilde{\chi}_q(n)\lambda_{P_r}(n)$.
    Then we have $f(I) = (+1, \ldots, +1)$ for some interval $I$ of length $k$. 
\end{lem}
\begin{proof}
    Without loss of generality, we assume that the length of $\tilde{\chi}_q$ is less than $k$. Let $[k]= \{1, \ldots, k\}$ and
    $I_k(n) =n + [k]= \{n+1, \ldots, n+k\}$ be an interval of length $k$ starting from $n+1 \in \N$ with $n+1 \equiv m$ mod $q$ for some $0 \leq m < q$. And let $\tilde{\chi}_q(I_k(n)) = (\tilde{\chi}_q(n+1), \ldots, \tilde{\chi}_q(n+k))$ denote the sign pattern of $\tilde{\chi}_q$ on $I_k(n)$. Suppose that $\tilde{\chi}_q$ only takes value of $-1$ at $n+J = \{n+a_1, \ldots, n+a_{r}\} \subset I_k(n)$ where $J = \{a_1, \ldots, a_r\}$. Let $p_1, \ldots, p_{r}$ be $r$ distinct primes greater than $k$ and coprime to $q$. By the Chinese remainder theorem, there exist solutions to the system below:
\begin{align*}
    n'+1 &\equiv m \text{ mod }q, \;\;\;  n'+a_j \equiv 0 \text{ mod }p_j \text{ for } j = 1, \ldots, r.
\end{align*}
Suppose that the solutions are in the form of $ n' = N + \alpha Q$ where $N \in \N$ is fixed, $\alpha \in \Z$, and $Q = q\prod\limits_{j = 1}^{r}p_j$. We can choose an $\alpha$ satisfying $p_j^{\nu_j} \parallel n'+a_j$ with $\nu_j$ odd for $j = 1, \ldots, r$ and $q^{n+1} | n'-n$. Then $\tilde{\chi}_q(I_k(n)) = \tilde{\chi}_q(I_k(N+\alpha Q)) = \tilde{\chi}_q(I_k(n'))$ by  \cite[Lemma 9.4]{klurman2023elliott}. Besides, $p_j \nmid n'+a_i$ for all $i \neq j$, and each $ 1\leq j \leq r$ since $p_j > k$ for $ 1\leq j \leq r$. Hence, we have
\begin{align*}
    f(n'+a_j) &= \tilde{\chi}_q(n'+a_j) = \tilde{\chi}_q(n+a_j) \text{ for } a_j \in [k]-J,\\
    f(n'+a_j) &= -\tilde{\chi}_q(p_j)^{\nu_j}\tilde{\chi}_q\Bigg(\frac{n'+a_j}{p_j^{\nu_j}}\Bigg) = -\tilde{\chi}_q(n'+a_j)=-\tilde{\chi}_q(n+a_j) = 1 \text{ for $a_j \in J$}.
\end{align*}
Hence, we have $f(m) = +1$ for all $m \in n' + [k]$. 
\end{proof}
\begin{lem}\lb{lem2.2}
    Let $Q >0$ and $\chi_q$ be a non-principal character mod $q$. Then we have
    \begin{align*}
        \frac{1}{\log Q}\sum\limits_{n \leq Q}\frac{\tilde{\chi}_q(n)}{n} = \frac{1}{\log Q}\prod\limits_{p | q}(1-\tilde{\chi}_q(p)p^{-1}) L(1, \chi_q) + O(qQ^{-1/8}).
    \end{align*}
\end{lem}
\begin{proof}
Let $T$ be a parameter that depends on $Q$, and $0 < c \leq 1/2$. By \cite[First effective Perron formula, II.2 on p.132]{tenenbaum2015introduction} and the residue theorem, we have
\begin{align*}
    \sum\limits_{n = 1}^{Q}\frac{\tilde{\chi}_q(n)}{n} &= \frac{1}{2\pi i}\int_{\frac{1}{\log Q} - iT}^{\frac{1}{\log Q} + iT}L(1+s, \tilde{\chi}_q)Q^s\frac{ds}{s} + R\\
    &= L(1, \tilde{\chi}_q) + I_1 + I_2 + I_3 + R,
\end{align*}
\begin{align*}
    \text{where } 
    I_1 &=\frac{1}{2\pi i}\int_{-c - iT}^{-c + iT}L(1+s, \tilde{\chi}_q)Q^s\frac{ds}{s},\;\;\;
    I_2=\frac{1}{2\pi i}\int_{-c + iT}^{\frac{1}{\log Q} + iT}L(1+s, \tilde{\chi}_q)Q^s\frac{ds}{s},\\
    I_3 &= \frac{1}{2\pi i}\int_{\frac{1}{\log Q} - iT}^{-c - iT}L(1+s, \tilde{\chi}_q)Q^s\frac{ds}{s},\text{ and }R = O\bv\sum\limits_{n=1}^\infty \frac{1}{n^{1+\frac{1}{\log Q}}(1 + T|\log(Q/n)|)}\bn. 
\end{align*}
Let $z=\sigma + it \in \C$. By  \cite[Lemma 10.15]{montgomery2007multiplicative}, 
\begin{align}
    |L(z, \chi_q)| \ll
    (1+(q(4+|t|))^{1-\sigma}\min \bv\frac{1}{|\sigma-1|}, \log q(4+|t|)\bn,  \; \text{ for } 0< \sigma \leq 2.\lb{bound}
\end{align}
Moreover, 
\begin{align}
    L(z, \tilde{\chi}_q) &= \prod\limits_{p | q}(1-\tilde{\chi}_q(p)p^{-z})  L(z, \chi_q) \lb{prodeq}
\end{align}
holds for $\sigma > 0$, since \ef{prodeq} holds when $\sigma > 1$, and it extends to the half-plane by analytic continuation since the right hand side of \ef{prodeq} is holomorphic.
Applying \ef{bound} and \ef{prodeq} to $I_1, I_2, I_3$, we obtain
\begin{align*}
    |I_1| &\ll \int_{-c - iT}^{-c + iT}|\prod\limits_{p | q}(1-\tilde{\chi}_q(p)p^{-(1+s)})| |L(1+s, \chi_q)||Q^s|\frac{|ds|}{|s|} \ll q^c(qT)^{c}Q^{-c}\log (qT), \\
    |I_2| &\ll \int_{-c + iT}^{\frac{1}{\log Q} + iT}|\prod\limits_{p | q}(1-\tilde{\chi}_q(p)p^{-(1+s)})| |L(1+s, \chi_q)||Q^s|\frac{|ds|}{|s|} \ll \frac{q^c(qT)^c\log (qT)}{T},\\
    |I_3| &\ll \int_{\frac{1}{\log Q} - iT}^{-c - iT}|\prod\limits_{p | q}(1-\tilde{\chi}_q(p)p^{-(1+s)})| |L(1+s, \chi_q)||Q^s|\frac{|ds|}{|s|} \ll \frac{q^c(qT)^c\log (qT)}{T}.
\end{align*}
We take $c = 1/2$ and $ T = \sqrt{Q}$, then 
\begin{align*} 
    \frac{1}{\log Q}\sum\limits_{n = 1}^{Q}\frac{\tilde{\chi}_q(n)}{n} &= \frac{1}{\log Q}\prod\limits_{p | q}(1-\tilde{\chi}_q(p)p^{-1}) L(1, \chi_q) + O(qQ^{-1/8}).
\end{align*}   
\end{proof}

\begin{proof}[Proof of Theorem 1.2] From the proof of Lemma \ref{lem2.1}, we can see that the number of the modified primes is independent of the choice of modified primes $p > k$ and the location of $-1$s in a sign pattern. It only depends on the number of $-1$s in an interval of length $k$. Let $\mc{I}_{q, k}(n)$ denote an interval of length $k$ such that $\tilde{\chi}_q(\mc{I}_{q, k}(n))$ contains the least number of $-1$s. $\delta_q(k)$ is equivalent to the number of $-1$s in $\tilde{\chi}_q(\mc{I}_{q, k}(n))$, for which
\begin{align*}
    \delta_q(k)= \frac{1}{2}(k - S_q(k)), \text{ where }S_q(k) = \sum\limits_{ m \in \mc{I}_{q, k}(n)} \tilde{\chi}_q (m)= \max\limits_{n \in \N \cup \{0\}} \sum\limits_{m = 1}^{k}\tilde{\chi}_q(n + m).
\end{align*}
In the following, we always assume the length of $\tilde{\chi}_q$ is less than $k$, otherwise, $\delta_q(k) = 0$ which is trivial. Next, 
we will estimate $\delta(k) = \max\limits_{q \in \N}\delta_q(k)$ by investigating its upper bound and lower bound. 
\subsection{The upper bound of $\delta(k)$} \lb{sec2.1}
We will show $\delta(k) \leq \frac{1}{2}k$.
\subsubsection{$\chi_q$ is the principal character mod $q$} Let $\chi_q$ be a principal character. In this case, it suffices to show
\begin{align*}
    A = \lim_{\alpha \to \infty}A(q^\alpha) =\lim_{\alpha \to \infty} \frac{1}{q^\alpha}\sum\limits_{n=1}^{q^\alpha}\sum\limits_{m=0}^{k-1}\tilde{\chi}_q(n+m) > 0,
\end{align*}
since $S_q(k) \geq A(q^\alpha)$ for all $ \alpha \in \N$.
Let $\alpha > k$ be an integer. After rearranging, we have
\begin{align*}
    A(q^\alpha) &= \frac{1}{q^\alpha}\Bigg(k\sum\limits_{n=1}^{q^\alpha}\tilde{\chi}_q(n) - \sum\limits_{m=1}^{k-1}(k-m)\tilde{\chi}_q(m) + \sum\limits_{m=1}^{k-1}(k-m)\tilde{\chi}_q(q^\alpha +m)\Bigg)=k\frac{1}{q^\alpha}\sum\limits_{n=1}^{q^\alpha}\tilde{\chi}_q(n),
\end{align*}
since $\tilde{\chi}_q(m) = \tilde{\chi}_q(q^\alpha +m)$ for all $1\leq m \leq (k-1)$ by \cite[Lemma 9.4]{klurman2023elliott}.
By \cite[Delange's Theorem, III.~4 on p.~326]{tenenbaum2015introduction}, $\tilde{\chi}_q$ possesses a positive mean value, so we obtain
\begin{align*}
   A = \lim_{\alpha \to \infty} A(q^\alpha) = k\lim_{\alpha \to \infty} \frac{1}{q^\alpha}\sum\limits_{n=1}^{q^\alpha}\tilde{\chi}_q(n) = \prod_{p|q}(1-p^{-1})\sum\limits_{\nu=0}^\infty\tilde{\chi}_q(p)^\nu p^{-\nu} > 0.
\end{align*}
\subsubsection{$\chi_q$ is non-principal mod $q$} 
Let $Q \geq 1$. By the definition of $\delta_q(k)$, we have
\begin{align}\lb{veryhellosum}
    \delta_q(k) \leq \frac{1}{2}\Bigg(k-\frac{1}{\log Q}\sum\limits_{n = 1}^{Q}\frac{1}{n}\sum\limits_{m = 0}^{k-1}\tilde{\chi}_q(n + m)\Bigg).
\end{align}
Then it suffices to prove that
\begin{align*}
    B = \frac{1}{\log Q}\sum\limits_{n = 1}^{Q}\frac{1}{n}\sum\limits_{m = 0}^{k-1}\tilde{\chi}_q(n + m)= o(1) \text{ as } Q\to \infty.
\end{align*}
After rearranging, we obtain
\begin{align*}
    B &= \frac{1}{\log Q} \Bigg(\sum\limits_{m = 0}^{k-1}\sum\limits_{n = 1}^{Q} \frac{\tilde{\chi}_q(n + m)}{n+m}+\sum\limits_{m = 0}^{k-1}m\sum\limits_{n = 1}^{Q}\frac{\tilde{\chi}_q(n + m)}{n(n+m)} \Bigg)\nonumber\\
    &=\frac{1}{\log Q} \Bigg(
    k\sum\limits_{n = 1}^{Q} \frac{\tilde{\chi}_q(n)}{n^{}}-\sum\limits_{n = 1}^{k-1} (k-n)\frac{\tilde{\chi}_q(n)}{n} + \sum\limits_{n = 1}^{k-1} (k-n)\frac{\tilde{\chi}_q(Q+n)}{(Q+n)} + O(k^2)\Bigg) \nonumber\\
    &= \frac{k}{\log Q} \sum\limits_{n = 1}^{Q} \frac{\tilde{\chi}_q(n)}{n} + O(\frac{k^2}{\log Q}). 
\end{align*}
If we choose $Q$ to be sufficiently large in terms of $k$ and $q$, we will obtain $B = o(1)$ by Lemma \ref{lem2.2}. As a result, taking $Q \to \infty$ we obtain $\delta(k) \leq \frac{1}{2}k$ by \ef{veryhellosum}.

\subsection{The lower bound of $\delta(k)$} We will bound $\delta(k)$ from below by $\delta_{q}(k)$ where $q$ is prime and $q < k$, as $\delta(k) \geq \delta_{q}(k)$ for all $q \in \N$. Let $q$ be a prime with $q<k$. Suppose $k = \sum\limits_{i=0}^\nu a_iq^i = k_0$ where $\nu$ is the largest integer such that $q^\nu \leq k$, $0\leq a_i < q$ for $i=0, \ldots, \nu$, and $a_\nu \neq 0$. Let $k_{j}$ denote the number of elements in $\mc{I}_{q, k_0}(n)$ that are divisible by $q^j$ i.e., $k_j = \sum\limits_{i = j}^\nu a_iq^{i-j} $.
When $q < k$, $-1$s in $\tilde{\chi}_q(\mc{I}_{q, k_0}(n))$ come from non-residues in a complete set of residue classes, non-residues in the incomplete residue classes, and the elements in $\mc{I}_{q, k}(n)$ divisible by $q$ of value $-1$ which can be seen as $-1$s in $\tilde{\chi}_q(\mc{I}_{q, k_1}(n))$. If $k_1 \geq q$, we count $-1$s in $\tilde{\chi}_q(\mc{I}_{q, k_1}(n))$ by replacing $\tilde{\chi}_q(\mc{I}_{q, k_0}(n))$ with $\tilde{\chi}_q(\mc{I}_{q, k_1}(n))$ and decompose these into complete, incomplete and divisible by $q$ sets, then replace $\tilde{\chi}_q(\mc{I}_{q, k_1}(n))$ by $\tilde{\chi}_q(\mc{I}_{q, k_2}(n))$ until $k_{\nu} < q$.

Let $I$ denote an incomplete set of length $l < q$ other than the last incomplete set of length $a_\nu$ and $0 \leq r \leq l$ be an integer. Suppose $I = I_1 \cup I_2$ where $I_1 = \{N_1q - r, N_1q-(r-1), \ldots, N_1q-1\}$ and $I_2 = \{N_2q + 1, N_2q + 2, \ldots, N_2q + (l-r)\}$ for some $N_1, N_2 \in \N$ with $N_1 < N_2$. We have
\begin{align*}
    \sum\limits_{n \in I}\chi_q(n) = \sum\limits_{m = 0}^{l}\chi_q(-r+m).
\end{align*} 
Then the minimal number of $-1$s in an incomplete set of total length $l$ (except the last incomplete set of length $a_\nu$), $\delta_q'(l)$ can be represented by
\begin{align*}
    \delta_q'(l) = \frac{1}{2}(l - S_q'(l)) \text{ where }S_q'(l) = \max\limits_{q-l\leq n \leq q} \sum\limits_{m = 0}^{l}\chi_q(n+m).
\end{align*}
Whereas, the minimal number of $-1$s in an incomplete set of length $a_\nu$ is $\delta_q(a_\nu)$, as the starting point of the last incomplete set is unrestricted.
Following the recurrence, for $i = 0, \ldots, \nu - 1$, we have
\begin{align}\lb{onesum}
    \#\{m \in \mc{I}_{q, k_i}: \tilde{\chi}_q(m) = -1\}&\geq \frac{1}{2}(q-1)k_i+\delta'_q(a_i) +  \#\{m \in \mc{I}_{q, k_{i+1}}: \tilde{\chi}_q(m) = -1\},
\end{align}
and when $i = \nu$,
\begin{align}\lb{twosum}
        \#\{m \in \mc{I}_{q, k_\nu}: \tilde{\chi}_q(m) = -1\}&\geq \delta_q(a_\nu).
\end{align}
By summing \ef{onesum} and \ef{twosum} over $i$, we obtain
\begin{align}
    \delta_q(k) 
    &\geq \sum\limits_{i = 1}^{\nu}\frac{1}{2}(q-1) k_i + \sum\limits_{j = 0}^{\nu-1}\delta_q'(a_i) +  \delta_q(a_\nu) \nonumber \\ 
   &\geq \frac{1}{2}(q-1) \sum\limits_{i = 1}^{\nu} \sum\limits_{j = i}^{\nu}a_jq^{j-i} +\frac{1}{2}\bv\sum\limits_{i=0}^\nu a_i - \sum\limits_{i = 0}^{\nu-1} S_q'(a_i) - S_q(a_\nu)\bn \nonumber\\ 
   &\geq\frac{1}{2}k - \frac{1}{2}\bv\sum\limits_{i = 0}^{\nu-1} S_q'(a_i) + S_q(a_\nu)\bn. \lb{delta3}
\end{align}
If we choose $q = 3$, then we have
\begin{align}\lb{for3}
    0 \leq S_3'(a_i) \leq 1 \text{ for } i = 0, \ldots, \nu-1 \text{ and } 1 \leq S_3(a_\nu) \leq 2 .
\end{align}
Applying \ef{for3} to \ef{delta3}, we obtain
\begin{align*}
    \delta_3(k) \geq \frac{1}{2}k - \frac{1}{2}\bv\sum\limits_{i = 0}^{\nu-1} 1 + 2\bn =  \frac{1}{2}k - \frac{1}{2}\bv \Big\lfloor\frac{\log k}{\log 3}\Big\rfloor + 2\bn = \frac{1}{2}k + O(\log k).
\end{align*}
This implies that $\delta(k) \geq \frac{1}{2}k + O(\log k)$. Combining with the upper bound from the previous section, we have $\delta(k) = \frac{1}{2}k + O(\log k)$.\\
\begin{rem}
    For the lower bound of $\delta_q(k)$ when $q < k$, one can bound \ef{delta3} by using the Pólya-Vinogradov inequality instead of choosing $q = 3$. Then the error term will be $O(\frac{\log k}{\log q} q^{1/2+o(1)}) = o(k)$ instead of $O(\log k)$.
\end{rem}

\end{proof}
\begin{proof}[Proof of Corollary 1.4]
Since $f$ is of length $k$, we have
\begin{equation}\lb{10}
    S = \sum\limits_{n\in K}\prod\limits_{j = 0}^{k}(1+f(n+j)) = 0 \text{ for any $K \subset \N$}.
\end{equation}
Suppose $K = (0, x]\cap \N \text{ with } x>0$ and $I = \{0, 1, \ldots, k\}$. Then expand \ef{10} and take $x \to \infty$, we have
\begin{align*}
     0= \lim_{x\to \infty}S &= \lim_{x\to \infty}(x + S_1(x) + S_2(x) + \cdots + S_{k+1}(x)),\\
    \text{where } S_i(x) &=\sum\limits_{\substack{I_i \subseteq I\\ |I_i| = i, 1\leq i \leq k+1}}\sum\limits_{n\leq x}\prod_{j\in I_i}f(n+j).
\end{align*}
By Conjecture \ref{conj1.3}, if $f$ is a non-pretentious function then $S_i(x)$ = $o(x)$ as $x\to \infty$ for $i = 1, \ldots, k+1$. Then we have
\begin{equation*}
 \lim_{x\to \infty}S =  x + o(x) \neq 0
\end{equation*}
which contradicts \ef{10}. As a consequence, $f$ must be pretentious to a twisted character $\chi(n)n^{it}$. In our case, we may assume $\chi$ is real and $t=0$, in other words, $f$ must be pretentious to a real primitive character or principal character $\chi$. Indeed, since if $\chi$ is not real, we have
\begin{align*}
    \mb{D}(f, \chi(n) n^{it};x) \geq \frac{1}{4} \sqrt{\log \log x} + O_{\chi}(1)
\end{align*}
by \cite[Lemma C.1]{matomaki2015averaged} and $\mb{D}(f, \chi(n) n^{it};x) \to \infty$ as $x \to \infty$ that contradicts $\mb{D}(f, \chi(n) n^{it};x) < \infty$. Also, by \ef{tri}, we have
\begin{align} \lb{pretentious}
    2\mb{D}(f, \chi(n) n^{it};x)\geq \mb{D}(1, \chi^2(n) n^{i2t};x) =\mb{D}(1, \chi_0(n) n^{i2t};x) =\log (1 + |2t|\log x) + O(1) 
\end{align}
and the right hand side of \ef{pretentious} tends to $\infty$ as $x \to \infty$. This contradicts $\mb{D}(f, \chi(n) n^{it};x) < \infty$, unless $|t| \ll 1/\log x$ as $x \to \infty$ which implies $t=0$. Moreover, according to the extension process, the size of \ef{defofJ} can be bounded by the upper bound of $\delta(k)$, for which
\begin{align*}
    |\mc{J}(k)| \leq \frac{1}{2}k.
\end{align*}

\end{proof}

\begin{acknowledgements}
    I would like to thank my supervisor Sacha Mangerel for fruitful discussions and constructive comments and Oleksiy Klurman for useful suggestions.
\end{acknowledgements}

\bibliographystyle{plain}
\addcontentsline{toc}{chapter}{Bibliography}
\bibliography{Signpattern}

\end{document}